\documentclass{amsart}

\usepackage[T1]{fontenc}
\usepackage{lmodern}
\usepackage{microtype}

\usepackage{amsthm,amsfonts,amssymb,amsmath,thmtools}
\usepackage{bm}
\usepackage{calligra}
\usepackage{mathrsfs}
\usepackage{mathtools}
\usepackage{stmaryrd}
\usepackage{tikz}
\usepackage{tikz-3dplot}
\usepackage{tikz-cd}
\usepackage[all]{xy}

\usepackage{bookmark}
\usepackage{enumitem}
\usepackage[margin=1.5in]{geometry}
\usepackage{indentfirst}

\usepackage{color}
\definecolor{dkgreen}{rgb}{0,0.4,0}
\definecolor{mblue}{rgb}{0,0.5,0.62}
\definecolor{gray}{rgb}{0.5,0.5,0.5}
\definecolor{burgendy}{rgb}{0.502, 0, 0.125}

\usepackage{hyperref}
\hypersetup{
    linktoc=page,
	colorlinks = true,
	linkcolor = mblue, 
	filecolor = cyan,
	urlcolor = dkgreen, 
	citecolor = burgendy, 
}

\usepackage[capitalize]{cleveref}

\theoremstyle{plain}
\newtheorem{thm}{Theorem}[section] 
\newtheorem{cor}[thm]{Corollary}
\newtheorem{lem}[thm]{Lemma}  
\newtheorem{prop}[thm]{Proposition}

\theoremstyle{definition}
\newtheorem{defn}[thm]{Definition}
\newtheorem{rmk}[thm]{Remark}
\newtheorem{exe}[thm]{Example}

\makeatletter
    \let\c@equation\c@thm
\makeatother
\numberwithin{equation}{section}

\creflabelformat{secnumber}{Section~#2#1#3}

\newcommand{\bb}[1]{\mathbb{#1}}
\newcommand{\rr}[1]{\mathrm{#1}}
\newcommand{\cc}[1]{\mathcal{#1}}
\newcommand{\scr}[1]{\mathscr{#1}}

\newcommand{\Hom}{\mathrm{Hom}}

\newcommand{\inv}{^{-1}}
\newcommand{\pr}{\mathrm{pr}}

\newcommand{\KEXP}{\mathrm{KExpVar}_k}

\let\emptyset\varnothing

\title[Invariance of irregular Hodge numbers]{On the invariance of irregular Hodge numbers under crepant birational equivalences}

\author{Yichen Qin} 
\email{yichen.qin@fudan.edu.cn}
\address{School of Mathematical Sciences, Fudan University, Handan Road 220, 200437, Shanghai, China}

\begin{document}

\begin{abstract}
The Batyrev--Kontsevich theorem asserts that birational Calabi--Yau varieties have the same Hodge numbers. In this article, we prove an analogue for Landau--Ginzburg models $(U,f)$, where $U$ is a smooth quasi-projective complex variety and $f$ is a regular function on $U$. Under a natural non-degeneracy assumption, we show that the classes of such models in the localized Grothendieck ring of complex algebraic varieties with exponentials are invariant under crepant birational equivalences. Consequently, the irregular Hodge numbers of the twisted de Rham cohomology $\rr{H}^k_{\rr{dR}}(U,f)$ are invariant as well.
\end{abstract}
\maketitle

\section{Introduction} 
Mirror symmetry predicts that a Calabi--Yau variety $X$ is associated with a mirror Calabi--Yau variety $X^\vee$, whose Hodge numbers satisfy
$$
h^{p,q}(X)=h^{\dim X-p,q}(X^\vee).
$$
In many mirror constructions, different choices may produce birational models of the mirror. This raises the question: do birational Calabi--Yau varieties have the same Hodge diamond? Batyrev used $p$-adic integration to prove that the Betti numbers of birational Calabi--Yau varieties are equal \cite{BatyrevBirationalCalabi99}. Kontsevich sketched in his Orsay lecture \cite{KontsevichLectureOrsay95} a proof using motivic integration. Later, the foundational work of motivic integration was completed by Denef and Loeser \cite{DenefLoeserGermsarcs99}, thereby proving the invariance of Hodge numbers.  
Moreover, Ito \cite{ito_bir-cy} and Wang \cite{wang_bir-cy} independently proved the same invariance by refining Batyrev's method with $p$-adic Hodge theory.

In recent years, mirror symmetry for non-Calabi--Yau varieties, and in particular for Fano varieties, has attracted considerable attention.  The mirror of a Fano variety is expected to be a Landau--Ginzburg model, namely a pair $(U,f)$ consisting of a smooth non-proper variety $U$ and a regular function $f\colon U\to\mathbb A^1_{\mathbb C}$. An important class of examples is provided by the Givental--Hori--Vafa mirrors of Fano complete intersections in smooth toric Fano varieties \cite{GiventalMirrorToric98,HoriVafa00}. In many cases, these mirrors admit Laurent polynomial presentations involving auxiliary choices, such as nef partitions and torus charts. The construction of such presentations and the birational relations among the resulting models, including mutations, have been studied in \cite{PrzyjalkowskiHoriVafa11,DoranHarderToricDegenerations16,AkhtarCoatesGalkinKasprzykMinkowski12,PrinceBirationality21}. See \cref{sec:Laurent-mutations-and-toric-Landau-Ginzburg-mirrors} for more discussion. Motivated by these results, it is natural to find an analogue of the Batyrev--Kontsevich theorem for Landau--Ginzburg models.

\subsection{Main result}\label{sec:main-result}
To formulate such an invariance statement, one first needs a suitable analogue of Hodge numbers for Landau--Ginzburg models. Katzarkov, Kontsevich and Pantev \cite{KatzarkovEtAlBogomolovTian17} proposed three candidates and conjectured that they coincide. See, for example, \cite{LUNTS2018189,Shamoto2018,sabbah2018properties,HarderHodgenumbers21} for discussions of their relations. Moreover, by the work of \cite{EsnaultEtAl$E_1$degenerationirregular17}, one of the candidates coincides with the irregular Hodge numbers. In this paper, we focus on irregular Hodge numbers. 

Irregular Hodge theory was initiated by Deligne \cite{DeligneTheorieHodge07} and developed further by many authors such as Esnault, Kontsevich, Mochizuki, Sabbah, and Yu; see, for example, \cite{SabbahFourierLaplace10,YuIrregularHodge14,EsnaultEtAl$E_1$degenerationirregular17,SabbahIrregularHodge18,MochizukiRescalability25}. Given a Landau--Ginzburg model $(U,f)$, we can define its \textit{twisted de Rham cohomologies} as the hypercohomologies of the twisted de Rham complex
    \begin{equation*} 
    \cc{O}_U\xrightarrow{~\rr{d}+\rr{d}f~}\Omega_U^1\xrightarrow{~\rr{d}+\rr{d}f~}\dots \xrightarrow{~\rr{d}+\rr{d}f~}\Omega_U^{\dim U},
    \end{equation*}
denoted by $\rr{H}^k_{\rr{dR}}(U,f)$. The \textit{irregular Hodge filtration} is a decreasing filtration on these cohomologies with desired Hodge-theoretic properties and the dimensions of the graded quotients are the \textit{irregular Hodge numbers}.

Instead of working with arbitrary Landau--Ginzburg models, we consider pairs with non-degenerate functions, which play the role of smooth projective varieties in this context; cf. \cref{defn:non-degenerate}. A \textit{pair with a non-degenerate function} is a triple $(X,D,f)$ consisting of a smooth projective variety $X$, a divisor $D$ with simple normal crossing support, and a regular function $f\colon X\setminus |D|\to \bb{A}^1$ satisfying a certain non-degeneracy condition in the sense of Katz \cite{KatzExponentialsums90} and Mochizuki \cite{MochizukiMixedTwistor15}. This class includes non-degenerate Laurent polynomials and, after translating the function by a regular value, the proper compactified Landau--Ginzburg models used in several verifications of the Katzarkov--Kontsevich--Pantev conjecture for Fano varieties; see \cref{exe:Laurent-Polynomial} and Remark~\ref{rmk:ghv-log-CY-compactifications}. In particular, the exponential mixed Hodge structure associated to $(X\setminus|D_X|,f)$ is pure; cf. \cref{prop:purity-EMHS}.

The birational relation we consider for pairs with non-degenerate functions is crepant equivalence. More precisely, we call two pairs with non-degenerate functions $(X,D_X,f)$ and $(Y,D_Y,g)$ \textit{crepant} if there exist birational morphisms $\varphi\colon Z\to X$ and $\psi\colon Z\to Y$ from a common smooth projective variety $Z$ such that
	$$\varphi^*(K_X+D_X)= \psi^*(K_Y+D_Y),$$
and the two functions $f\circ \varphi$ and $g\circ \psi$ agree on a Zariski dense open subset of $Z$. This notion is closely related to the usual $K$-equivalence. When the functions extend to morphisms to $\bb{P}^1$, the crepant condition above reduces to the usual $K$-equivalence;   see \cref{defn:crepant} and \cref{rmk:proper-crepant-k-equivalence}.

Our main result is the following.

\begin{thm}\label{intro::main}
   Let $(X,D_X,f)$ and $(Y,D_Y,g)$ be crepant pairs with non-degenerate functions. Then the Landau--Ginzburg models $(X\setminus |D_X|,f)$ and $(Y\setminus |D_Y|,g)$ have the same irregular Hodge numbers. 
\end{thm}

We give several geometric consequences of \cref{intro::main} in \cref{sec:Laurent-mutations-and-toric-Landau-Ginzburg-mirrors}. First, the irregular Hodge numbers of proper compactified Landau--Ginzburg models are invariant under flops over $\bb{P}^1$. We then prove a corresponding invariance statement for compatible birational log Calabi--Yau models. Finally, in the toric setting, we apply the result to mutation-equivalent Laurent polynomials whenever they admit suitable non-degenerate log Calabi--Yau compactifications. These compactifications occur in many Givental--Hori--Vafa mirrors, although their existence is an additional geometric assumption.

\begin{rmk}
	For a fixed pair $(X,D_X)$, Zhang and the author showed in \cite[Thm.~1.3.1]{QIn26} that the irregular Hodge numbers of $(X\setminus|D_X|,f)$ are independent of the choice of non-degenerate function $f$. Nevertheless, when comparing two different pairs, the compatibility of the functions required in \cref{defn:crepant} is essential for \cref{intro::main}; crepancy of the underlying pairs alone is not sufficient. See \cref{rmk:not-extendable}.
\end{rmk}

\subsection{Outline of the proof}
Let $\mathcal{E}xp\mathcal{M}_{\bb{C}}$ denote the localized Grothendieck ring of complex algebraic varieties with exponentials. A complex algebraic variety $V$ equipped with a regular function $h\colon V\to\bb{A}^1$ determines a class $[V,h]_{\bb{C}}$ in this ring; see \cite[\S1.1]{Chambert-LoirLoeserMotivicheight16} and \cref{sec:motivic-integrations}.
The key input is the following stronger motivic statement.
\begin{thm}[{\ref{thm::main}}]\label{intro::motivic}
If $(X,D_X,f)$ and $(Y,D_Y,g)$ are crepant pairs with non-degenerate functions, then
\begin{equation}\label{eq:intro-motivic-equality}
[X\setminus|D_X|,f]_{\bb{C}}
=
[Y\setminus|D_Y|,g]_{\bb{C}}
\end{equation}
in $\mathcal{E}xp\mathcal{M}_{\bb{C}}$.
\end{thm}

Applying the motivic measure defined by the irregular Hodge polynomial to \eqref{eq:intro-motivic-equality} gives equality of the corresponding irregular Hodge polynomials.
Purity of the exponential mixed Hodge structures associated with non-degenerate functions then allows one to recover the individual irregular Hodge numbers; see \cref{cor::same-Hodge-numbers}.
Thus, \cref{intro::motivic} implies \cref{intro::main}.
We now outline the proof of \cref{intro::motivic} in two steps.
\medskip

\noindent\textit{Step 1.}
Assume that $f$ and $g$ extend to morphisms from $X$ and $Y$ to $\mathbb{P}^1$. In this case, following Kontsevich's motivic integration strategy \cite{KontsevichLectureOrsay95}, we apply the change of variables formula for constructible exponential functions \cite[Thm.~4.2.1]{CluckersLoeserConstructibleexponential10} to obtain \eqref{eq:intro-motivic-equality}; see \cref{prop:crepant-mor}.

\medskip

\noindent\textit{Step 2.}
In general, the pole divisor of a non-degenerate function may meet its zero locus transversally. In this case, the motivic change-of-variables formula on a common resolution only compares the common open part and gives an identity of the form
$$
[\varphi(W),f]_{\bb C}=[\psi(W),g]_{\bb C},
$$
where $W=\varphi^{-1}(X\setminus|D_X|) \cap\psi^{-1}(Y\setminus|D_Y|).$ It does not directly compare the complements of this common open subset.

To overcome this difficulty, starting from a non-degenerate function $f_0$ on $(X_0,D_0)$, we construct a sequence of crepant blow-ups
$$
\pi_i\colon
(X_i,D_i,f_i)\longrightarrow
(X_{i-1},D_{i-1},f_{i-1}),
\qquad i=1,\ldots,n,
$$
along the intersections of the zero locus with components of the pole divisor. After finitely many steps $Z(f_n)$ is disjoint from $|D_n|$, and hence $f_n$ extends to a morphism $X_n\to\mathbb P^1$.

The essential point is that these crepant blow-ups need not preserve the open Landau--Ginzburg space. By \cref{prop:crepant-minus-trash}, we have
$$
[\pi_i^{-1}(X_{i-1}\setminus|D_{i-1}|),f_i]_{\bb C}
=
[X_{i-1}\setminus|D_{i-1}|,f_{i-1}]_{\bb C}.
$$
By \cref{lem:crepant-blow-up}, the open set $X_i\setminus|D_i|$ either equals $\pi_i^{-1}(X_{i-1}\setminus|D_{i-1}|)$ or contains in addition an exceptional piece $E_i^\circ\simeq\mathbb A^1_t\times Z_i$, on which the function is of the form $f_i|_{E_i^\circ}=t$. Since the exponential class $[\mathbb A^1_t\times Z_i,t]_{\bb C}$ vanishes, every step preserves the exponential motivic class:
$$
[X_{i-1}\setminus|D_{i-1}|,f_{i-1}]_{\bb C}
=
[X_i\setminus|D_i|,f_i]_{\bb C}.
$$
Applying this construction to both sides produces crepant modifications which remain crepant to each other. Step~1 therefore applies and proves \cref{intro::motivic}.

\subsection{Organization}
This article is arranged as follows. In \cref{sec:crepant}, we recall the definitions of simple normal crossing pairs, non-degenerate functions, and crepant equivalences between them, and give some basic properties and examples. In \cref{sec:motivic-integrations}, we recall the necessary results of motivic integration for exponential Grothendieck groups following \cite{Chambert-LoirLoeserMotivicheight16,BiluMotivicEuler23}, and show that the irregular Hodge polynomial is a motivic measure. In \cref{sec:proof}, we prove that two crepant pairs with non-degenerate functions define the same class in $\mathcal{E}xp\mathcal{M}_{\bb{C}}$, hence proving \cref{intro::main}. Finally, in \cref{sec:Laurent-mutations-and-toric-Landau-Ginzburg-mirrors}, we give geometric applications to flops, log Calabi--Yau models, and mutations of toric Landau--Ginzburg models.

\section{Crepant equivalences}\label{sec:crepant}
In this section, we recall simple normal crossing pairs, non-degenerate functions on them, and crepant equivalences between them.

\subsection{Simple normal crossing pairs and crepant morphisms}
Let $X$ be a smooth projective variety over $\bb{C}$ and $D$ an effective divisor on $X$. The pair $(X, D)$ is called a \textit{simple normal crossing pair} if $D$ has simple normal crossing support. 

Two pairs $(X,D_X)$ and $(Y,D_Y)$ are called \textit{crepant} if there exist birational morphisms $\varphi\colon Z\to X$ and $\psi\colon Z\to Y$ from a common smooth projective variety $Z$ such that
$$
\varphi^*(K_X+D_X)=\psi^*(K_Y+D_Y).
$$
Let $\pi\colon X\to Y$ be a birational morphism. We say that $\pi$ defines a \textit{crepant morphism of pairs} $\pi\colon (X,D_X)\to (Y,D_Y)$ if
$$
\pi^*(K_Y+D_Y)= K_X+D_X.
$$
Equivalently, this is the special case of the above definition with $Z=X$ and one of the two maps the identity.
 
\begin{exe}\label{exe:toric}
For a Laurent polynomial $f$ on a torus, any smooth toric compactification $X_f$ with toric boundary $D_f$ is log-Calabi--Yau, i.e.,
$$
	K_{X_f}+D_f= 0.
$$
For any two such smooth toric compactifications $X_f$ and $X_g$ of Laurent polynomials on the same torus, there exists a smooth projective variety $Z$ with birational morphisms $Z\to X_f$ and $Z\to X_g$. Since both pairs are log-Calabi--Yau, these morphisms give a crepant equivalence between $(X_f,D_f)$ and $(X_g,D_g)$.
\end{exe}

\begin{exe}\label{exe:crepant}
	Let $X=\bb{P}^2$ with coordinates $x_0,x_1,x_2$, and $D=L_0\cup L_1\cup L_2$, where the three lines $L_0, L_1$ and $L_2$ are defined by $x_0,x_1$, and $x_2$ respectively.  Take a point $p$ in $L_2\setminus (L_0\cup L_1)$. Let $\tilde{X}=\rr{Bl}_{p}\bb{P}^2$ be the blow-up of $X$ along $p$, $\tilde{L}_2$   the strict transform of $L_2$, and $\tilde{D}$ the divisor $L_0\cup L_1\cup \tilde{L}_2$.  Then the blow-up morphism $\pi\colon (\tilde{X},\tilde{D})\to (X,D)$ is crepant.
\end{exe}

\begin{exe}\label{exe:crepant-2}
	Let $X=\bb{P}^3$ with coordinate $[x_0:x_1:x_2:x_3]$. Let $D=L_0\cup L_1$ where two lines $L_0$ and $L_1$ are defined by $x_0$ and $x_1$ respectively. Take $p=[0:0:1:1]\in L_0\cap L_1$. We write $\tilde{X}$ as the blow-up of $X$ at $p$, and let $\tilde{D}=\tilde{L}_0+\tilde{L}_1$, where $\tilde{L}_0$ and $\tilde{L}_1$ are the strict transforms, and $E$ be exceptional divisor. Then the blow-up morphism $\pi\colon (\tilde{X},\tilde{D})\to (X,D)$ is crepant.
\end{exe}

\subsection{Non-degenerate functions}
We recall the notion of non-degenerate functions of Katz and Mochizuki following \cite{QIn26}.
\begin{defn}\label{defn:non-degenerate}
    A regular function $f$ on a smooth quasi-projective variety $U$ is called \textit{non-degenerate}  if 
    \begin{enumerate}
        \item there exists a smooth projective compactification $X$ of $U$ and a divisor $D=\sum_{i=1}^r m_i D_i$ with simple normal crossing support, such that $U$ is the complement of the support of $D$ in $X$;
        \item there exist a global section $s_0\in \Gamma(X,\cc{O}_X(D))$ and global sections $\sigma_i\in \Gamma(X,\cc{O}_X(D_i))$ such that $f$ agrees with the function
            $$\frac{s_0}{s_\infty}\colon X\setminus Z(\sigma_1^{m_1}\dots \sigma_r^{m_r})\to \bb{A}^1,$$
        where $s_\infty=\sigma_1^{m_1}\dots \sigma_r^{m_r}$ and $Z(s_0)$ is transversal to $Z(\sigma_1^{m_1}\dots \sigma_r^{m_r})$.
    \end{enumerate}
    For simplicity,  we call such a triple $(X,D,f)$ a \textit{pair with non-degenerate function}.
\end{defn}

Unlike the tame compactified Landau--Ginzburg models in \cite{KatzarkovEtAlBogomolovTian17}, we require neither the Calabi--Yau condition $K_X+D=0$ nor that $f$ extend to a morphism $X\to \mathbb{P}^1$.  

\begin{exe}\label{exe:Laurent-Polynomial}
		One major source of non-degenerate functions is from non-degenerate Laurent polynomials. By \cite[\S~4]{Batyrev93}, a Laurent polynomial $f$ on a torus is non-degenerate if and only if its zero locus in the toric compactification $X_f$ determined by the Newton polytope $\Delta(f)$ is transversal to the boundary divisor $D_f$. When $f$ is convenient, i.e., $0$ is in the interior of the Newton polytope of $f$, the pole divisor $P_f$ has support $D_f$. Hence, a convenient non-degenerate Laurent polynomial is non-degenerate on $(X_f,P_f)$ in the sense of \cref{defn:non-degenerate}. 

		Other examples arise naturally from mirror symmetry. Let $f\colon U\to\mathbb A^1$ admit a smooth projective compactification $\bar f\colon X\to\mathbb P^1$ such that $D=\bar f^*(\infty)$ has simple normal crossing support and $U=X\setminus|D|$. If $c$ is a regular value of $f$, then $\bar f^{-1}(c)$ is smooth and disjoint from $D$. Hence $(X,D,f-c)$ is a pair with a non-degenerate function, and replacing $f$ by $f-c$ changes neither the twisted de Rham complex nor its irregular Hodge numbers. In particular, the proper Landau--Ginzburg models used to verify the Katzarkov--Kontsevich--Pantev conjecture for del Pezzo surfaces and Fano threefolds fit into our setting \cite{LUNTS2018189,HarderHodgenumbers21,CP25}; see also Remark~\ref{rmk:ghv-log-CY-compactifications}.
\end{exe}

However, we don't allow horizontal divisors in \cref{defn:non-degenerate}, i.e., the pole divisor of $f$ is $D$. In fact, pairs with non-degenerate functions are analogous to smooth projective varieties. For example, the associated exponential mixed Hodge structures are pure, cf. \cref{prop:purity-EMHS}. 

For a general Landau--Ginzburg model $(U,f)$, by resolution of singularities, any Landau--Ginzburg model $(U,f)$ admits a smooth projective compactification $(X,D)$ such that $D=X\setminus U$ is a simple normal crossing divisor and $Z(f)$ intersects $D$ transversally. In general, $D=D_h+ D_v$ can be decomposed as the horizontal divisor and the pole divisor of $f$, and $f$ extends to a function $F$ on $X-D_v$. Then, the triple  $(X,D_v,F)$ is a pair with non-degenerate functions.

\begin{defn}\label{defn:associated-compactification}
	A \textit{non-degenerate compactification} of a Landau--Ginzburg model $(U,f)$ is a pair with non-degenerate functions $(X,D_v,F)$ such that $U$ is a Zariski open subset of $X$ with complement $D=D_v+D_h$ a divisor with simple normal crossing support, and $F|_{U}=f$. If $K_X+D_v=0$, we call it a \textit{Calabi--Yau} compactification.
\end{defn}

\begin{exe}\label{exe:quadric}
	Consider the convenient but degenerate Laurent polynomial $f=y+\frac{(1+x)^2}{xy}$ on the two-dimensional torus $T$. In the toric compactification $X_f$, the zero locus $Z(f)$ intersects $P_f=X_f-T$ non-transversally. Blowing up the base locus twice gives a triple $(X,D=D_h+D_v,F)$ with $X-D=T$ and $F|_T=f$. Then $(X,D_v,F)$ is a non-degenerate compactification of $(T,f)$. Setting $U=X\setminus|D_v|$, we have $\dim\mathrm{H}^2_{\rr{dR}}(U,F)=4$, whereas $\dim\mathrm{H}^2_{\rr{dR}}(T,f)=3$.
\end{exe}

\subsection{Crepant equivalences of pairs with non-degenerate functions}
Next, we introduce the notion of crepant equivalence for pairs with non-degenerate functions.
 
\begin{defn}\label{defn:crepant}
	We say two pairs with non-degenerate functions $(X,D_X,f)$ and $(Y, D_Y,g)$ are \textit{crepant} if there exist birational morphisms $\varphi: Z\to X$ and $\psi\colon Z\to Y$ from a common smooth projective variety $Z$ such that $(X,D_X)$ and $(Y,D_Y)$ are crepant, 
  	and the two functions $f\circ \varphi$ and $g\circ \psi$ agree on a Zariski dense open subset of $Z$.  
	
	A morphism $\pi\colon X\to Y$ is a \textit{crepant morphism of pairs with non-degenerate functions} if it induces a crepant morphism of pairs $\pi\colon (X,D_X)\to (Y,D_Y)$ and $f=g\circ \pi$ on a Zariski dense open subset of $X$. 
\end{defn}

\begin{rmk}\label{rmk:proper-crepant-k-equivalence}
	Suppose that $f$ and $g$ extend to morphisms $\bar f\colon X\to\bb{P}^1$ and $\bar g\colon Y\to\bb{P}^1$, respectively. Then $D_X=\bar f^*(\infty)$ and $D_Y=\bar g^*(\infty)$. If the functions are compatible on a common smooth projective model $Z$ as in \cref{defn:crepant}, then
	$$
	\varphi^*D_X=(\bar f\circ\varphi)^*(\infty)
	=(\bar g\circ\psi)^*(\infty)=\psi^*D_Y.
	$$
	Consequently,
	$$
	\varphi^*(K_X+D_X)=\psi^*(K_Y+D_Y)
	\quad\Longleftrightarrow\quad
	\varphi^*K_X=\psi^*K_Y.
	$$
	Thus, for compactified Landau--Ginzburg models compatible over $\bb{P}^1$, the crepant condition in \cref{defn:crepant} is precisely the usual $K$-equivalence of $X$ and $Y$.
\end{rmk}

\begin{exe}\label{exe:an-example}
We take the crepant morphism $\pi\colon (\tilde{X},\tilde{D})\to (X,D)$ from \cref{exe:crepant} and look at the non-degenerate functions.

For $(X,D)$, the interior $U:=X\setminus|D|$ is the torus $\bb{G}_{m}^2$. The non-degenerate functions $f$ on $(X,D)$ are of the form 
    $$\frac{F}{x_0x_1x_2}\colon U\to \bb{A}^1$$
where $F$ is a cubic homogeneous polynomial in $x_0,x_1$, and $x_2$ whose zero locus $Z(F)$ is transverse to $D$.

For $(\tilde{X},\tilde{D})$, the interior $\tilde{U}:=\tilde{X}\setminus|\tilde{D}|$ is the  union $\bb{G}_{m}^2\cup \bb{A}^1$. The non-degenerate functions $\tilde{f}$ on $(\tilde{X},\tilde{D})$ are of the form 
    $$\frac{\tilde F}{x_0x_1\tilde{x}_2}\colon \tilde{U}\to \bb{A}^1$$
where $\tilde F$ is the pullback of a cubic homogeneous polynomial $F$ such that $Z(F)$ is transverse to $D$ and passes through $p$. 

One can verify that $(\tilde{X},\tilde{D},\tilde{f})$ is crepant to $(X,D,f)$ if $f$ and $\tilde {f}$ are defined using the same cubic homogeneous polynomials.
\end{exe}

Notice that, in contrast to \cref{exe:crepant}, the crepant morphism in \cref{exe:crepant-2} cannot be upgraded to a crepant equivalence of pairs with non-degenerate functions. 

\begin{prop}\label{prop:not-extendable}
	For the crepant blow-up morphism $\pi\colon (\tilde{X},\tilde{D})\to (X,D)$ in \cref{exe:crepant-2}, there does not exist non-degenerate functions $f$ and $\tilde{f}$ on $(X,D)$ and $(\tilde{X},\tilde{D})$ respectively such that $(\tilde{X},\tilde{D},\tilde{f})$ is crepant to $(X,D,f)$. 
\end{prop}
\begin{proof}
	For $(X,D)$, the interior $U:=X\setminus|D|=\mathbb{G}_{m}\times \mathbb{A}^2$.  A non-degenerate function $f$ on $(X,D)$ is of the form 
		$$\frac{F}{x_0x_1}\colon U\to \bb{A}^1$$
	where $F$ is a quadratic homogeneous polynomial such that $Z(F)$ is transverse to $D$.

	For $(\tilde{X},\tilde{D})$, the interior $\tilde U:=\tilde X\setminus|\tilde D|=(\mathbb{G}_{m}\times \mathbb{A}^2)\cup (\mathbb{G}_{m}\times \mathbb{A}^1)$. A non-degenerate function $\tilde f$ on $(\tilde{X},\tilde{D})$ is of the form 
	$$\frac{\tilde F}{\tilde{x}_0\tilde{x}_1}\colon \tilde{U}\to \bb{A}^1$$
	where $\tilde F$ is the pullback of a quadratic homogeneous polynomial $F$ such that $Z(F)$ is transverse to $D$ and passes through $p$.

	 If there exist $f$ and $\tilde f$ such that $(\tilde{X},\tilde{D},\tilde{f})$ is crepant to $(X,D,f)$, then $\tilde f$ would agree with $\pi^*f$ on the dense open set $\tilde X\setminus(\tilde D\cup E)\cong X\setminus(D\cup\{p\})$, hence $\tilde f=\pi^*f$ as rational functions. Write $f=F/(x_0x_1)$ with $F$ quadratic and let $m=\operatorname{mult}_p(F)$. Since $\pi^*L_i=\tilde L_i+E$, we have
$$
\operatorname{div}(\pi^*f)=\widetilde{Z(F)}-\tilde L_0-\tilde L_1+(m-2)E.
$$
For $\pi^*f$ to be regular on $\tilde X\setminus\tilde D$, the coefficient of $E$ in the polar part must vanish, forcing $m\ge 2$. However, transversality of $Z(F)$ to $D$ at $p$ means $p$ is a smooth point of the quadric $Z(F)$, so $m=1$. Thus $\pi^*f$ has a simple pole along $E$ and is not regular on the new open stratum $E\setminus(\tilde L_0\cup \tilde L_1)\cong\mathbb G_m\times\mathbb A^1$ inside $\tilde X\setminus\tilde D$. No such $\tilde f$ exists.
\end{proof}

\begin{rmk}\label{rmk:not-extendable}
	Because of \cref{prop:not-extendable}, our main theorem does not apply to \cref{exe:crepant-2}. By \cite[Thm.~1.1.1 \& Thm.~1.1.3]{QIn26}, the irregular Hodge numbers of $(X\setminus|D_X|,f)$ and $(\tilde{X}\setminus|D_{\tilde{X}}|,\tilde{f})$ depend only on the underlying pairs. They can be computed from $\rr{H}^i(U,f\inv(t))$ and $\rr{H}^i(\tilde{U},\tilde{f}\inv(t))$, respectively, for $t$ sufficiently large.
	
	One can show that $\rr{H}^3(U,f\inv(t))$ is an extension of $\bb{Q}(-1)$ and $\bb{Q}(-2)$, while $\rr{H}^3(\tilde{U},\tilde{f}\inv(t))$ vanishes. Thus the compatibility of the functions in \cref{intro::main} cannot be omitted.
\end{rmk}

\section{Recollections on motivic integration}\label{sec:motivic-integrations}
In this section, we gather some basic definitions and properties of motivic integration, following \cite{Chambert-LoirLoeserMotivicheight16,BiluMotivicEuler23}.

\subsection{The Grothendieck group of varieties with potentials}

\subsubsection{Over a field}
Let $k$ be a field of characteristic $0$. A $k$-variety will be a separated $k$-scheme of finite type. 

\begin{defn}\label{defn::KExpVark}
The \textit{Grothendieck group of varieties with potentials} $\mathrm{KExpVar}_k$ (cf. \cite{Chambert-LoirLoeserMotivicheight16}) is the group generated by pairs $(U,f)$ consisting of a $k$-variety $U$ and a regular function $f\colon U\to \bb{A}^1_k$, subject to the relations
\begin{itemize}
	\item[(i).] $(U,f)=(V,f\circ u)$ for $u\colon V\to U$ an isomorphism of $k$-varieties;
	\item[(ii).] $(U,f)=(V,f\mid_V)+(Z,f\mid_Z)$ for $Z$ a closed subscheme of $U$ and $V=U\setminus Z$ its complement;
	\item[(iii).] $(U\times_k\bb{A}^1_k,\pr_2)=0$ for every $k$-variety $U$.
\end{itemize}
\end{defn}
We will denote the class of a pair $(U,f)$ in $\KEXP$ by $[U,f]_{k}$. We may further equip $\mathrm{KExpVar}_k$ with a ring structure by setting
	\begin{equation}\label{eq::product_in_KExpVark}
		[U,f]_{k}\cdot [V,g]_{k}:=[U\times V,f\boxplus g]_{k},
	\end{equation}
where $f\boxplus g:=f\circ \pr_1+g\circ\pr_2$ is the Thom--Sebastiani sum.
As in the case of the Grothendieck ring of $k$-varieties, we denote by $\mathbb{L}$ the class $[\bb{A}^1_k,0]$. Let $S=\{\mathbb{L}, \mathbb{L}^n-1\mid n\geq 1\}$. Then the localization $\mathrm{KExpVar}_k[S\inv]$ will be denoted by $\mathcal{E}xp\mathcal{M}_k$.

We define the ring $\KEXP^{sm,qproj}$ in a way similar to \cref{defn::KExpVark}, by considering only pairs $(U,f)$ with $X$ smooth and quasi-projective. There is a natural ring homomorphism
	$$\varphi^{sm,qproj}\colon\KEXP^{sm,qproj} \to \KEXP.$$

\begin{lem}\label{lem::sm-qproj}
	The natural homomorphism $\varphi^{sm,qproj}$ is an isomorphism.
\end{lem}
\begin{proof}
	Apart from keeping track of the regular functions, the proof is identical to the one in the classical case; see \cite[Ch.~2, Prop.~2.6.5]{Chambert-LoirEtAlMotivicIntegration18}.
\end{proof}

\subsubsection{Over a general base}

The Grothendieck group $\mathrm{KExpVar}_k$ introduced above is the absolute case, corresponding to $S=\operatorname{Spec} k$. In motivic integration, however, one needs to consider families of varieties over a base $S$; in particular, the Grothendieck groups of varieties over jet spaces $\scr{L}_n(X)$ will play an essential role later.  

Let $S$ be a $k$-variety. The \textit{Grothendieck group of $S$-varieties with potentials} $\mathrm{KExpVar}_S$ is defined in the same way as \cref{defn::KExpVark}, except that every $k$-variety is replaced by an $S$-variety. The class of a pair $(U\to S,f)$ is denoted by $[U,f]_S$, and the product is given by the Thom--Sebastiani sum, exactly as in \eqref{eq::product_in_KExpVark}.

The relative theory is functorial in the base. Let $\varphi\colon S\to T$ be a morphism of $k$-varieties. On the one hand, viewing an $S$-variety as a $T$-variety via $\varphi$ gives the \textit{pushforward} (or forgetful) homomorphism
\begin{align*}
	\varphi_!\colon \mathrm{KExpVar}_S&\to \mathrm{KExpVar}_T,\\
	[U,f]_S&\mapsto [U,f]_T.
\end{align*}
On the other hand, pulling back a $T$-variety along $\varphi$ gives the \textit{pullback} homomorphism
\begin{align*}
	\varphi^*\colon \mathrm{KExpVar}_T&\to \mathrm{KExpVar}_S,\\
	[V,g]_T&\mapsto [V\times_T S,g\circ \pr_1]_S.
\end{align*}
These two operations will be used constantly when integrating motivic residual functions over arc spaces and applying the change of variable formula.

\subsection{Motivic measure} 
Let $A$ be a ring. An $A$-valued \textit{motivic measure} on $\mathrm{KExpVar}_S$ is a ring homomorphism 
	$$\mu\colon \mathrm{KExpVar}_S\to A$$
such that 
\begin{enumerate}
	\item $\mu$ maps the three kinds of relations in \cref{defn::KExpVark} to $0$;
	\item $\mu([U,f]_S\cdot [V,g]_S)=\mu([U,f]_S)\cdot \mu([V,g]_S)$.
\end{enumerate}
The aim of this subsection is to construct a motivic measure taking values in $K_0(\rr{EMHS})$, where $\rr{EMHS}$ is the category of exponential mixed Hodge structures introduced by Kontsevich and Soibelman \cite{KontsevichSoibelmanCohomologicalHall11} as an irregular analogue of Deligne's mixed Hodge structures.

\subsubsection{Exponential mixed Hodge structures}

Let $k$ be $\mathbb{C}$. For any complex algebraic variety $X$, there is an abelian category $\mathrm{MHM}(X)$ of algebraic mixed Hodge modules, whose bounded derived category admits the six-functor formalism; see \cite{SaitoMixedHodge90} for details. The \textit{exponential mixed Hodge structures} are mixed Hodge modules on the affine line whose global cohomology vanishes, as introduced by Kontsevich and Soibelman in \cite{KontsevichSoibelmanCohomologicalHall11}; see also \cite[Appx.]{FresanEtAlHodgetheory22}. Concretely, let $\pi\colon \mathbb{A}^1\to\rr{Spec}\,\mathbb{C}$ be the structure morphism and define the full subcategory
$$
\mathrm{EMHS}\subset \mathrm{MHM}(\mathbb{A}^1)
$$
consisting of those $\mathcal{M}\in\mathrm{MHM}(\mathbb{A}^1)$ with $\pi_*\mathcal{M}=0$. There is an exact idempotent endofunctor
\begin{equation}\label{eq:Pi-functor}
    \Pi\colon \mathrm{MHM}(\mathbb{A}^1)\to\mathrm{MHM}(\mathbb{A}^1),
\end{equation}
defined by
\begin{equation}\label{eq:Pi}
    \Pi(\mathcal{M}) \;:=\; \operatorname{sum}_*\bigl(j_!\mathbb{Q}_{\mathbb{G}_m}^{\mathrm{H}}\boxtimes\mathcal{M}\bigr)
    \;=\; \mathcal{M}\star(j_!\mathbb{Q}_{\mathbb{G}_m}^{\mathrm{H}}),
\end{equation}
where $j\colon\mathbb{G}_m\hookrightarrow\mathbb{A}^1$ is the inclusion, $\operatorname{sum}\colon\mathbb{A}^1\times\mathbb{A}^1\to\mathbb{A}^1$ is addition, and $\star$ is the additive convolution, defined as $M\star N:=\rr{sum}_*(M\boxtimes N)$.

The functor $\Pi$ is an exact projector onto $\mathrm{EMHS}$: its essential image equals $\mathrm{EMHS}$, and $\Pi(\mathcal{N})=\mathcal{N}$ for $\mathcal{N}\in\mathrm{EMHS}$. In particular, any constant Hodge module on $\mathbb{A}^1$ is annihilated by $\Pi$. 

The de Rham fiber functor is an exact functor
\begin{equation}\label{eq:de_rham_fiber}
	\Xi_{\rr{dR}}\colon \mathrm{EMHS}\to \rr{Vect}_{\bb{C}},
    \quad 
    \mathcal{M}\mapsto\rr{H}^1_{\rr{dR}}(\bb{A}^1,\mathcal{M}\otimes \mathcal{E}^t)
\end{equation}
where $\mathcal{E}^t$ is the exponential $\scr{D}_{\bb{A}^1}$-module $(\cc{O},\rr{d}+\rr{d}t)$. Alternatively, it sends $\mathcal{M}$ to the fiber at $1\in \bb{A}^1$ of the Fourier transform of $\mathcal{M}$. 

Let $X$ be an algebraic variety of dimension $n$. For a regular function $g\colon X\to \bb{A}^1$ and an integer $i$, we consider the following exponential mixed Hodge structures:
\begin{equation}\label{eq:EMHS-functions}
	\mathrm{H}^i(X,g):=\Pi(\mathcal{H}^{i-n}g_*\bb{Q}_X^{\rr{H}}),
    \quad
     \mathrm{H}^i_c(X,g):=\Pi(\mathcal{H}^{i-n}g_!\bb{Q}_X^{\rr{H}}).
\end{equation}
The de Rham fiber of $\mathrm{H}^i(X,g)$ is $\mathrm{H}^i_{\rr{dR}}(X,g)$, and the de Rham fiber of $\mathrm{H}^i_c(X,g)$ is $\mathrm{H}^i_{\rr{dR},c}(X,g)$. 

We equip the category $\KEXP$ of regular functions with the product induced by additive convolution. Then we have the following:

\begin{prop}\label{prop::motivic-measure}
	There is a ring homomorphism 
		$$\mu\colon \KEXP\to \rr{K}_0(\rr{EMHS}).$$
\end{prop}

\begin{proof}

The candidate map is defined as 
$$[U,f]_{k}\mapsto \sum_i (-1)^i[\mathrm{H}^i_c(U,f)],$$
where $\mathrm{H}^i_c(U,f)$ are the exponential mixed Hodge structures defined in \eqref{eq:EMHS-functions}.
By \cref{lem::sm-qproj}, it suffices to verify for smooth quasi-projective varieties $U$ and $V$, and regular functions $f$ on $U$ and $g$ on $V$, that
\begin{itemize}
	\item[(i).] $\mu([U,f]_k)=\mu([V,f\circ u]_k)$ for any isomorphism $u\colon V\to U$;
	\item[(ii).] $\mu([U,f]_k)-\mu([Z,f\mid_Z]_k)-\mu([V,f\mid_V]_k)=0$ for a closed subscheme $Z$ of $U$ and $V=U\setminus Z$;
	\item[(iii).] $\mu([U\times_k\bb{A}^1_k,\pr_2]_k)=0$ for any $k$-variety $U$;
	\item[(iv).] $\mu([U\times_k V,f\boxplus g]_k)=\mu([U,f]_k)\cdot \mu([V,g]_k)$.
\end{itemize}

For (i), the isomorphism $u$ induces $u_!\bb{Q}^{\rr{H}}_V\simeq\bb{Q}^{\rr{H}}_U$, hence 
$$(f\circ u)_!\bb{Q}^{\rr{H}}_V\simeq f_!\bb{Q}^{\rr{H}}_U,$$ so the associated EMHS are isomorphic.

For (ii), use the localization triangle
$$
	j_!\bb{Q}^{\rr{H}}_V\to \bb{Q}^{\rr{H}}_U\to i_!\bb{Q}^{\rr{H}}_Z\xrightarrow{+1}.
$$
Applying $\Pi\circ f_!$ and taking cohomology yields the long exact sequence
$$
	\cdots \to \cc{H}^i\Pi (f|_V)_!\bb{Q}^{\rr{H}}_V \to \cc{H}^i\Pi f_!\bb{Q}^{\rr{H}}_U \to \cc{H}^i\Pi (f|_Z)_!\bb{Q}^{\rr{H}}_Z \to \cc{H}^{i+1}\Pi (f|_V)_!\bb{Q}^{\rr{H}}_V\to \cdots,
$$
which gives 
$$\mu([U,f]_k)-\mu([Z,f\mid_Z]_k)-\mu([V,f\mid_V]_k)=0$$ 
in $\rr{K}_0(\rr{EMHS}).$

For (iii), write $d=\dim(U\times_k\bb{A}^1_k)=\dim U+1$. It suffices to show that each object $$M_r:=\Pi\cc{H}^{r-d}(\pr_2)_!\bb{Q}^{\rr{H}}_{U\times\bb{A}^1}$$ is zero. Since the de Rham fiber functor $\Xi_{\rr{dR}}$ is exact and faithful, it is enough to check that the de Rham fiber of the dual of $M_r$ vanishes. That dual is $$\Pi\cc{H}^{d-r}(\pr_2)_*\bb{Q}^{\rr{H}}_{U\times\bb{A}^1},$$ whose de Rham fiber is $$\rr{H}^{2d-r}_{\rr{dR}}(U\times\bb{A}^1,\pr_2).$$ By \cite[Lem.~3.1(i)]{YuIrregularHodge14}, all the twisted de Rham cohomologies of $(U\times\bb{A}^1,\pr_2)$ vanish; hence $M_r=0$ for every $r$.

For (iv), the K\"unneth formula gives
$$
	(f\times g)_!\bb{Q}^{\rr{H}}_{U\times V}\simeq(f_!\bb{Q}^{\rr{H}}_U)\boxtimes (g_!\bb{Q}^{\rr{H}}_V),
$$
and applying $\mathrm{sum}_!$ yields
$$
	(f\boxplus g)_!\bb{Q}^{\rr{H}}_{U\times V}\simeq(f_!\bb{Q}^{\rr{H}}_U)\star_! (g_!\bb{Q}^{\rr{H}}_V),
$$
where $\star_!$ is the additive convolution, defined as $M\star_!N:=\rr{sum}_!(M\boxtimes N)$.

Recall that $\Pi$ is itself the $*$-convolution with $j_!\mathbb{Q}_{\mathbb{G}_m}^{\mathrm{H}}$, as in \eqref{eq:Pi}, and we have $\Pi\circ\Pi=\Pi$. Notice that for any objects $A,B\in\mathrm{MHM}(\mathbb{A}^1)$, the cone of $A\star_! B\to A\star_* B$ is a constant Hodge module, hence annihilated by $\Pi$. Moreover, $\Pi$ commutes with $*$-convolution because $*$-convolution is associative. Therefore, we obtain
$$
	\Pi(f\boxplus g)_!\bb{Q}^{\rr{H}}_{U\times V}\simeq \Pi(  f_!\bb{Q}^{\rr{H}}_U\star   g_!\bb{Q}^{\rr{H}}_V)
	\simeq(\Pi f_!\bb{Q}^{\rr{H}}_U)\star (\Pi g_!\bb{Q}^{\rr{H}}_V).
$$

The additive convolution $\star$ restricts to an exact bifunctor on $\mathrm{EMHS}$ and induces the product on $K_0(\mathrm{EMHS})$. Consequently,
\begin{equation*}
	\begin{split}
		\mu([U\times V,f\boxplus g]_k)
		&=[\Pi(f\boxplus g)_!\bb{Q}^{\rr{H}}_{U\times V}]\\
		&=[(\Pi f_!\bb{Q}^{\rr{H}}_U)\star  (\Pi g_!\bb{Q}^{\rr{H}}_V)]\\
		&=\sum_{i,j}(-1)^{i+j}[(\cc{H}^i\Pi f_!\bb{Q}^{\rr{H}}_U)\star  (\cc{H}^j\Pi g_!\bb{Q}^{\rr{H}}_V)]\\
		&=\sum_{i,j}(-1)^{i+j}[\cc{H}^i\Pi f_!\bb{Q}^{\rr{H}}_U]\cdot [\cc{H}^j\Pi g_!\bb{Q}^{\rr{H}}_V]\\
		&=[\Pi f_!\bb{Q}^{\rr{H}}_U]\cdot[\Pi g_!\bb{Q}^{\rr{H}}_V]\\
		&=\mu([U,f]_k)\cdot \mu([V,g]_k).
		\qedhere
	\end{split}
\end{equation*}
\end{proof}

Each exponential mixed Hodge structure $M$ has a weight filtration induced from that of $\mathrm{MHM}(\mathbb{A}^1)$, defined as
    $$W_k^{\rr{EMHS}} M:=\Pi(W_k^{\mathrm{MHM}} M),$$
where $W_\bullet^{\mathrm{MHM}}$ is the weight filtration in the category $\mathrm{MHM}(\mathbb{A}^1)$. In general, the weight filtration $W^{\rr{EMHS}}_\bullet$ is different from the weight filtration in $\mathrm{MHM}(\mathbb{A}^1)$. Unless otherwise specified, the weight filtration $W_\bullet$ on an exponential mixed Hodge structure $M$ in this paper always refers to $W_\bullet^{\rr{EMHS}}$.
The exponential mixed Hodge structures $\rr{H}^i(X,g)$, $\rr{H}_c^i(X,g)$ are mixed of weights at least $i$ and mixed of weights at most $i$ respectively by \cite[A.19]{FresanEtAlHodgetheory22}. In particular, for a proper function $g$, we have $\rr{H}^i(X,g)=\rr{H}_c^i(X,g)$ pure of weight $i$.

In \cite[Prop. 5.2.1]{QIn26}, we proved the following result:

\begin{prop}\label{prop:purity-EMHS}
	Let $f$ be a non-degenerate function with respect to $(X,D)$ and set $U=X\setminus|D|$. Then the exponential mixed Hodge structures $\mathrm{H}^i(U,f)$ and $\mathrm{H}^i_c(U,f)$ are isomorphic and pure of weight $i$.
\end{prop}

\subsubsection{The irregular Hodge polynomial}
Let $k$ be $\bb{C}$. The de Rham fiber $\Xi_{\rr{dR}}M$ of an object $M$ in $\mathrm{EMHS}$ carries an \textit{irregular Hodge filtration} $F_{\rr{irr}}^\bullet \Xi_{\rr{dR}}M$, constructed using a
generalization of Deligne's filtration \cite[\S6]{SabbahFourierLaplace10}; see also   \cite[Prop. 2.61 \& Def. 3.2]{SabbahIrregularHodge18}. 

For exponential mixed Hodge structures coming from geometry, the de Rham fibers of $\mathrm{H}^i_{?}(X,g)$ are isomorphic to $\mathrm{H}^i_{\mathrm{dR},?}(X,g)$ for $?\in \{\emptyset,c\}$. In this case, Esnault, Sabbah, and Yu showed in \cite[\S1]{EsnaultEtAl$E_1$degenerationirregular17} that the irregular Hodge filtration on the de~Rham fiber coincides with the Kontsevich filtration and the Yu filtration \cite{YuIrregularHodge14} on the twisted de Rham cohomologies $\mathrm{H}^i_{\mathrm{dR},?}(X,g)$.

As an analogue of the classical Hodge polynomial, we define the \textit{irregular Hodge polynomial} of $M$ as
\begin{equation}\label{eq:irr-HP}
	\rr{IrrHP}(M):=\sum_{k}(-1)^k\sum_{\alpha\in \bb{Q} }\dim  \mathrm{gr}_{F^\bullet_{irr}}^\alpha\Xi_{\rr{dR}} (\rr{gr}^W_k M)s^kt^{\alpha}\in \bb{Z}[s^{\pm 1},t^{\frac{1}{n}},n\in\bb{Z}\backslash\{0\}].
\end{equation}

\begin{lem}\label{lem:irrhp-ring}
	$\rr{IrrHP}$ is a ring homomorphism.
\end{lem}
\begin{proof}
	It suffices to show that $\rr{IrrHP}(M\otimes N)=\rr{IrrHP}(M)\cdot \rr{IrrHP}(N)$. One can use \cite[Prop. 3.54]{SabbahIrregularHodge18} to identify an exponential mixed Hodge structure as an irregular mixed Hodge structure, and then use the compatibility of the irregular Hodge filtration with the weight filtration and tensor product \cite[Prop.3.18]{SabbahIrregularHodge18} to conclude.
\end{proof}

Hence, combining with the motivic measure 
	$$\mu\colon \KEXP\to \rr{K}_0(\rr{EMHS}),$$
one deduces a motivic measure 
	$$\rr{IrrHP}\colon \KEXP\to\bb{Z}[s^{\pm 1},t^{\frac{1}{n}},n\in\bb{Z}\backslash\{0\}].$$

\begin{rmk}
	For a Landau--Ginzburg model $(U,f)$, the irregular Hodge polynomial $\rr{IrrHP}(U,f)$ lies in $\bb{Z}[s^{\pm 1},t^{1/N}]$ for some $N$ depending on $(U,f)$. More precisely, take a good compactification $X$ of $U$ such that $X\setminus U$ is a simple normal crossing divisor, the pole divisor $P$ of $f$ is supported on $X\setminus U$, and $f$ extends to a morphism $X\to \bb{P}^1$. Then $N$ is the least common multiple of the multiplicities of the irreducible components of $P$.
\end{rmk}

After inverting the irregular Hodge polynomials of $\{\mathbb{L},\mathbb{L}^n-1\mid n\geq 1\}$, we obtain the following morphism.

\begin{cor}\label{cor::motivic-measure}
The irregular Hodge polynomial induces a morphism 
	$$\rr{IrrHP}\colon \mathcal{E}xp\mathcal{M}_k\to \bb{Z}[s^{\pm 1},t^{\pm \frac{1}{n}},((s^2t)^n-1)\inv, n \geq 1].$$
\end{cor}

\begin{cor}\label{cor::same-Hodge-numbers}
	If $[U,f]_{\mathbb{C}}=[V,g]_{\mathbb{C}}$ in $\mathcal{E}xp\mathcal{M}_{\mathbb{C}}$, and all the associated exponential mixed Hodge structures $\mathrm{H}^i_c(U,f)$ and $\mathrm{H}^i_c(V,g)$ are pure of weight $i$, then $(U,f)$ and $(V,g)$ have the same irregular Hodge numbers.
\end{cor}
\begin{proof}
	By the definition of localization, there exists $q=\mathbb{L}^a \cdot \prod_{i=1}^k (\mathbb{L}^{b_i}-1)\in S$ such that $q([U,f]_{\mathbb{C}}-[V,g]_{\mathbb{C}})=0$ in $\KEXP$. In particular, the irregular Hodge polynomial of $q$ is nonzero. By \cref{cor::motivic-measure}, we have 
	$$\mathrm{IrrHP}(q)\cdot (\mathrm{IrrHP}(U,f)-\mathrm{IrrHP}(V,g))=0,$$
	where  $\mathrm{IrrHP}(q),\mathrm{IrrHP}(U,f) $ and $\mathrm{IrrHP}(V,g)$ belong to the integral ring $\mathbb{Z}[s^{\pm 1},t^{\pm \frac{1}{N}}]$ for some $N\geq 1$. Hence 
	$\mathrm{IrrHP}(U,f)=\mathrm{IrrHP}(V,g)$.

	As the associated exponential mixed Hodge structures $\mathrm{H}^i_c(U,f)$ and $\mathrm{H}^i_c(V,g)$ are pure of weight $i$, one can recover the irregular Hodge numbers from the irregular Hodge polynomial. More precisely, the irregular Hodge number $h^\alpha(U,f,i)$ is the coefficient of $s^it^\alpha$. So $(U,f)$ and $(V,g)$ have the same irregular Hodge numbers.
\end{proof}

\subsection{Motivic integration}

Let $X$ be a smooth variety over $k$. For each integer $n\geq 0$, the $n$-th jet scheme $\scr{L}_n(X)$ represents the functor
$$
	\scr{L}_n(X)(Y)=\Hom_k(Y\times_k k[t]/(t^{n+1}),X)
$$
on $k$-schemes $Y$. The arc scheme $\scr{L}(X)=\varprojlim_n\scr{L}_n(X)$ represents
$$
	\scr{L}(X)(Y)=\Hom_k(Y\times_k k[[t]],X).
$$
For $m\geq n$, the truncation morphism $\theta_{n,X}^m\colon\scr{L}_m(X)\to\scr{L}_n(X)$ is induced by the natural surjection $k[t]/(t^{m+1})\to k[t]/(t^{n+1})$. Pullback along truncation gives maps
$$
	(\theta_{n,X}^m)^*\colon\mathcal{E}xp\mathcal{M}_{\scr{L}_n(X)}\to\mathcal{E}xp\mathcal{M}_{\scr{L}_m(X)}.
$$

A \textit{motivic residual function} on $X$ is an element of the direct limit
$$
	\varinjlim_m\mathcal{E}xp\mathcal{M}_{\scr{L}_m(X)}.
$$
Concretely, if a motivic residual function $h$ is represented by a pair $[H,f]_{\scr{L}_m(X)}$, where $H$ is a scheme over $\scr{L}_m(X)$ for some $m$ and $f\colon H\to\bb{A}^1$ is a regular function, its motivic integral over $\scr{L}(X)$ is defined by
\begin{equation}\label{eq:motivic-integral}
	\int_{\scr{L}(X)}h\,\mathrm{d}\mu_X
	:=\bb{L}^{-\dim X\,(m+1)}\cdot[H,f]_k\in\mathcal{E}xp\mathcal{M}_k,
\end{equation}
where $[H,f]_k$ denotes the class of the pair $[H,f]$ in $\mathcal{E}xp\mathcal{M}_k$ obtained by forgetting the $\scr{L}_m(X)$-structure.

We end this subsection by recalling the change of variables formula \cite[Thm.\,4.2.1]{CluckersLoeserConstructibleexponential10} in a form that we need. Following the notation of \cite[\S~4.2]{CluckersLoeserConstructibleexponential10}, $\mathrm{ordjac}\,\varphi$ denotes the order of the Jacobian of $\varphi$. We write $\mathrm{ord}_E$ for the corresponding order function associated with a divisor $E$. For a birational morphism between smooth varieties, $\mathrm{ordjac}\,\varphi=\mathrm{ord}_{K_{Y/X}}$.

\begin{thm}[Change of variables]
	Let $\varphi\colon Y\to X$ be a proper birational morphism between smooth varieties over $k$, and also denote by $\varphi\colon\scr{L}(Y)\to\scr{L}(X)$ the induced morphism on arc schemes. Then for any motivic residual function $h\in\mathcal{E}xp\mathcal{M}_{\scr{L}(X)}$,
	\begin{equation*}
	\int_{\scr{L}(X)} h \, \mathrm{d}\mu_X = \int_{\scr{L}(Y)} \varphi^*h\cdot \bb{L}^{-\mathrm{ordjac}\,\varphi} \, \mathrm{d}\mu_Y.
	\end{equation*}
\end{thm}

\section{Invariance of irregular Hodge numbers}\label{sec:proof}
In this section, we prove \cref{thm::main} and apply it to mutations of Laurent polynomials.
\subsection{Proof of \cref{thm::main}}

\begin{prop}\label{prop:crepant-mor}
	Let $(X,D_X,f)$ and $(Y,D_Y,g)$ be crepant pairs with non-degenerate functions such that $f$ and $g$ extend to proper morphisms $X\to\bb{P}^1_{\bb{C}}$ and $Y\to\bb{P}^1_{\bb{C}}$, respectively. Then $[U,f]_{\bb{C}}=[V,g]_{\bb{C}}$ in $\mathcal{E}xp\mathcal{M}_{\bb{C}}$, where $U=X\setminus|D_X|$ and $V=Y\setminus|D_Y|$.
\end{prop}

\begin{proof}
Choose birational morphisms $\varphi\colon Z\to X$ and $\psi\colon Z\to Y$ such that
$$
  g\circ\psi=f\circ\varphi
  \quad\text{and}\quad
  \varphi^*(K_X+D_X)=\psi^*(K_Y+D_Y).
$$
Since $f$ and $g$ extend to morphisms $\tilde{f}$ and $\tilde{g}$ from $X$ and $Y$ to $\mathbb{P}^1$ respectively, we have $U=\tilde{f}\inv(\bb{A}^1)$ and $V=\tilde{g}\inv(\bb{A}^1)$ by non-degenerateness. Set $W:=\psi^{-1}(V)=\varphi^{-1}(U)$ and $F:=(f\circ\varphi)|_W=(g\circ\psi)|_W$.

Consider the residual functions $[V ,g]_Y$ on $\scr L(Y)$ and $[U,f]_X$ on $\scr L(X)$. By the definition of the motivic integral of a residual function and the change of variables formula applied to $\psi$, we have
$$
  \bb{L}^{-\dim V}[V,g]_{\bb{C}}
  =\int_{\scr L(Y)}[V,g]_Y\,\mathrm d\mu_Y
  =\int_{\scr L(Z)}[W,F]_Z\,
    \bb{L}^{-\mathrm{ord}_{K_{Z/Y}}}\,\mathrm d\mu_Z.
$$
Similarly, applying the change of variables formula to $\varphi$ gives
$$
  \bb{L}^{-\dim U}[U,f]_{\bb{C}}
  =\int_{\scr L(X)}[U,f]_X\,\mathrm d\mu_X
  =\int_{\scr L(Z)}[W,F]_Z\,
    \bb{L}^{-\mathrm{ord}_{K_{Z/X}}}\,\mathrm d\mu_Z.
$$

The crepant equality can be rewritten as
$$
  \psi^*D_Y-K_{Z/Y}=\varphi^*D_X-K_{Z/X}.
$$
Since $W=\psi^{-1}(V)=\varphi^{-1}(U)$ is disjoint from both $\psi^*D_Y$ and $\varphi^*D_X$, the orders of these divisors vanish on $\scr L(W)$. Hence
$$
  \mathrm{ord}_{K_{Z/Y}}=\mathrm{ord}_{K_{Z/X}}
  \quad\text{on }\scr L(W),
$$
and the two integrands above coincide. Therefore $\bb{L}^{-\dim V}[V,g]_{\bb{C}}=\bb{L}^{-\dim U}[U,f]_{\bb{C}}$, which yields the desired equality because $\dim U=\dim V$.
\end{proof}

\begin{prop}\label{prop:crepant-minus-trash}
	Given a crepant morphism $\pi\colon(X,D_X,f)\to (Y,D_Y,g)$ between pairs with non-degenerate functions, we have  $[\pi^{-1}V,f]_{\bb{C}}=[V,g]_{\bb{C}}$ in $\mathcal{E}xp\mathcal{M}_{\bb{C}}$, where $U=X\setminus|D_X|$ and $V=Y\setminus|D_Y|$.
\end{prop}
\begin{proof}
Let  $W:=\pi^{-1}(V)$.
Since $\pi\colon(X,D_X)\to(Y,D_Y)$ is crepant, we have
$$
  \pi^*(K_Y+D_Y)=K_X+D_X.
$$
Writing $K_{X/Y}=K_X-\pi^*K_Y$, this can be rewritten as
$$
  \pi^*D_Y-K_{X/Y}=D_X.
$$

By the definition of the motivic integral of the residual function $[V,g]_Y$ and the change of variables formula applied to $\pi$,
$$
  \bb{L}^{-\dim V}[V,g]_{\bb{C}}
  =\int_{\scr L(Y)}[V,g]_Y\,\mathrm d\mu_Y
  =\int_{\scr L(X)}[W,f|_W]_X\,
    \bb{L}^{-\mathrm{ord}_{K_{X/Y}}}\,\mathrm d\mu_X.
$$

On $\scr L(W)$ we have $\mathrm{ord}_{\pi^*D_Y}=0$ because $W=\pi^{-1}(V)$ is disjoint from $\pi^*D_Y$. Therefore the crepant equality gives
$$
  -\mathrm{ord}_{K_{X/Y}}
  =-\mathrm{ord}_{K_{X/Y}}+\mathrm{ord}_{\pi^*D_Y}
  =\mathrm{ord}_{D_X}
  =0
  \quad\text{on }\scr L(W),
$$
where the last equality holds because $W$ is also disjoint from $D_X$. Hence
$$
  \bb{L}^{-\dim V}[V,g]_{\bb{C}}
  =\int_{\scr L(X)}[W,f|_W]_Z\,\mathrm d\mu_X
  =\bb{L}^{-\dim U}[W,f|_W]_{\bb{C}}
  =\bb{L}^{-\dim U}[\pi^{-1}V,f]_{\bb{C}}.
$$
Since $\dim U=\dim V$, the result follows.
\end{proof}

\begin{lem}\label{lem:crepant-blow-up}
	Let $f$ be a non-degenerate function on $(X,D_X)$. Write $D_X=\sum_{i=1}^r n_iD_i$, and assume that $Z:=Z(f)\cap D_1$ is non-empty. Let $\tilde{X}$ be the blow-up of $X$ at $Z$, $\tilde{D}$ the strict transform of $D_X$, and $\tilde{f}$ the pullback of $f$ to $\tilde{X}$. Then $(\tilde{X},\tilde{D}+(n_1-1)E,\tilde{f})$ is crepant to $(X,D_X,f)$ via the blow-up map. Moreover,
		\begin{enumerate}
			\item $\tilde{f}$ has pole orders $n_1$ and $n_1-1$ along the strict transform of $D_1$ and $E$ respectively, and $Z(\tilde{f})$ intersects $E$ but not the strict transform of $D_1$;
			\item if $n_1=1$, then $\tilde f$ has no pole along $E$, and
		$$
			[E\setminus|\tilde D|,\tilde f]_{\bb C}=0.
		$$
	\end{enumerate}
\end{lem}
\begin{proof}
Since $Z(f)$ intersects $D_X$ transversally, $Z$ is smooth of codimension two. We have
$$
	K_{\tilde X}=\pi^*K_X+E,
	\qquad
	\pi^*D_1=\tilde D_1+E,
	\qquad
	\pi^*D_i=\tilde D_i\quad(i\neq1).
$$
Hence
$$
	\pi^*D_X=\tilde D+n_1E,
$$
and therefore
$$
\begin{aligned}
	K_{\tilde X}+\tilde D+(n_1-1)E
	&=(\pi^*K_X+E)+(\pi^*D_X-n_1E)+(n_1-1)E\\
	&=\pi^*(K_X+D_X).
\end{aligned}
$$
Thus $\pi$ is crepant.

We compute locally near a point $p\in Z$. Let
$$
	I(p):=\{j\neq1\mid p\in D_j\}.
$$
By transversality, there are local coordinates
$x,y,(z_j)_{j\in I(p)},(z_k)_{k\not\in I(p)}$ such that
$$
	D_1=\{x=0\},\qquad Z(f)=\{y=0\},\qquad
	D_j=\{z_j=0\}\quad(j\in I(p)),
$$
and
$$
	f=
	\frac{y}
	{x^{n_1}\prod_{j\in I(p)}z_j^{n_j}}
	u(x,y,z_j,z_k),
$$
where $u$ is a unit.

In the chart $y=xv$, one has
$$
	\tilde f=
	\frac{v}
	{x^{n_1-1}\prod_{j\in I(p)}z_j^{n_j}}
	u(x,xv,z_j,z_k),
$$
while in the chart $x=yw$, one has
$$
	\tilde f=
	\frac{1}
	{w^{n_1}y^{n_1-1}\prod_{j\in I(p)}z_j^{n_j}}
	u(yw,y,z_j,z_k).
$$
These formulas show that the pole orders along $\tilde D_1$ and $E$ are $n_1$ and $n_1-1$, respectively. They also show that $Z(\tilde f)$ meets $E$ along $\{x=v=0\}$ in the first chart, but does not meet $\tilde D_1$. They further show that the new boundary has simple normal crossing support and that $Z(\tilde f)$ is transversal to it.

Suppose now that $n_1=1$. By transversality, the natural map
$$
	N_{Z/D_1}\oplus N_{Z/Z(f)}\longrightarrow N_{Z/X}
$$
is an isomorphism, and
$$
	N_{Z/Z(f)}\cong N_{D_1/X}|_Z.
$$
With the convention that $\mathbb P(N_{Z/X})$ parametrizes lines, removing the section
$$
	\tilde D_1\cap E=\mathbb P(N_{Z/D_1})
$$
from $E=\mathbb P(N_{Z/X})$ gives the total space of
$$
	\mathcal L
	:=
	N_{Z/D_1}\otimes N_{D_1/X}^{-1}|_Z
	\cong\mathcal O_X(Z(f)-D_1)|_Z.
$$
Moreover, for $j\neq1$,
$$
	E\cap\tilde D_j=\pi^{-1}(Z\cap D_j),
$$
and consequently
$$
	E^\circ\cong\mathbb V(\mathcal L|_{Z^\circ}),
$$
where $E^\circ =E-|\tilde{D}|$ and $Z^\circ=Z-\cup_{j\neq 1}|D_j|$.
The restriction $\tilde f|_{E^\circ}$ is fibrewise linear, and hence corresponds to a section
$$
	\sigma\in\Gamma(Z^\circ,\mathcal L^\vee).
$$
Indeed, in a local fibre coordinate $v$, the preceding formula becomes
$$
	\tilde f|_{E^\circ}
	=
	a\,v,
	\qquad
	a=
	\frac{u(0,0,z_j,z_k)}
	{\prod_{j\in I(p)}z_j^{n_j}}
	\in\mathcal O^\times_{Z^\circ}.
$$
Thus $\sigma$ is nowhere vanishing and trivializes $\mathcal L^\vee$. Under the induced trivialization of $\mathcal L$, the function $\tilde f$ is the standard fibre coordinate. Therefore
$$
	(\pi|_{E^\circ},\tilde f):
	E^\circ\xrightarrow{\sim}Z^\circ\times\mathbb A^1,
$$
and hence
$$
	[E^\circ,\tilde f]_{\bb C}
	=
	[Z^\circ\times\mathbb A^1,\pr_2]_{\bb C}
	=0.
$$
\end{proof}
\begin{prop}\label{prop:crepant-blow-up}
	Given a pair with non-degenerate functions $(X,D_X,f)$, there exists a crepant morphism $\pi\colon (\tilde{X},D_{\tilde{X}},\tilde{f})\to (X,D_X,f)$ such that $\tilde{f}$ extends to a proper morphism $\tilde{X}\to \bb{P}^1_{\bb{C}}$ and
	$$
		[\tilde{X}\setminus |D_{\tilde{X}}|,\tilde{f}]_{\bb{C}}=[X\setminus |D_X|,f]_{\bb{C}}
	$$
	in $\mathcal{E}xp\mathcal{M}_{\bb{C}}$.
\end{prop}
\begin{proof}
	Assume that the pole divisor of $f$ is $\sum_{i\in I} a_iD_i$, where the $D_i$ are the irreducible components of $D_X$, and let $I'\subseteq I$ be the set of indices such that $Z(f)$ intersects $D_i$. We argue by induction on the total multiplicity $N:=\sum_{i\in I'}a_i$. When $N=0$, $Z(f)$ does not intersect $D$. So we can take $\pi=\mathrm{id}$ and there is nothing to prove.

	If there exists $i_0\in I'$ such that $a_{i_0}>1$, we blow-up $X$ along $Z(f)\cap D_{i_0}$ to obtain a new pair $(X',D_{X'},f')$ with
	$$
		D_{X'}:=\widetilde{D_X}+(a_{i_0}-1)E,
	$$
	where $\widetilde{D_X}$ is the strict transform of $D_X$ and $E$ is the exceptional divisor. By \cref{lem:crepant-blow-up}, this $(X',D_{X'},f')$ is crepant to $(X,D_X,f)$. Moreover, the pole orders of $f'$ along the strict transform of $D_{i_0}$ and $E$ are $a_{i_0}$ and $a_{i_0}-1$, respectively, and $Z(f')$ intersects $E$ but not the strict transform of $D_{i_0}$. Since $E\subseteq |D_{X'}|$, we have $$X'\setminus |D_{X'}|=\pi^{-1}(X\setminus |D_X|),$$ hence \cref{prop:crepant-minus-trash} gives
	$$
		[X'\setminus |D_{X'}|,f']_{\bb{C}}=[X\setminus |D_X|,f]_{\bb{C}}
	$$
	in $\mathcal{E}xp\mathcal{M}_{\bb{C}}$. Therefore $Z(f')$ intersects $D_j$ for $j\in I'\setminus\{i_0\}$ and $E$, and the total multiplicity is $$N-a_{i_0}+(a_{i_0}-1)=N-1.$$ Thus we reduce to the pair $(X',D_{X'},f')$ and may apply the induction hypothesis.
	
	If all intersecting components have multiplicity $1$, i.e., $a_i=1$ for all $i\in I'$, choose $i_0\in I'$ and blow-up $X$ along $Z(f)\cap D_{i_0}$ to obtain $(X',D_{X'},f')$, where $D_{X'}$ is the strict transform of $D_X$. By \cref{lem:crepant-blow-up}, the blow-up $$\pi\colon (X',D_{X'},f')\to (X,D_X,f)$$ is crepant and $f'$ is regular along $E$. We have
	$$
		X'\setminus |D_{X'}|=\pi^{-1}(X\setminus |D_X|)\cup (E\setminus |D_{X'}|).
	$$
	Hence, by \cref{prop:crepant-minus-trash},
	$$
		[X'\setminus |D_{X'}|,f']_{\bb{C}}=[X\setminus |D_X|,f]_{\bb{C}}+[E^\circ,f']_{\bb{C}},
	$$
	where $E^\circ=E\setminus |D_{X'}|$. By \cref{lem:crepant-blow-up}, $[E^\circ,f']_{\bb{C}}$ vanishes in $\mathcal{E}xp\mathcal{M}_{\bb{C}}$. Hence, we again reduce to the case of $(X',D_{X'},f')$ with total multiplicity $N-1$ and we can apply the induction hypothesis. 
\end{proof}

\begin{thm}\label{thm::main}
	Given two pairs with non-degenerate functions $(X,D_X,f)$ and $(Y,D_Y,g)$ that are crepant to each other, they define the same class in $\mathcal{E}xp\mathcal{M}_{\bb{C}}$.
\end{thm}
\begin{proof}
	Let $(X,D_X,f)$ and $(Y,D_Y,g)$ be two non-degenerate triples that are crepant to each other. Then there exists a smooth proper variety $Z$ and birational morphisms $\varphi\colon Z\to X$ and $\psi\colon Z\to Y$ such that
	$$
		\varphi^*(K_X+D_X)=\psi^*(K_Y+D_Y)
	$$
	and $f\circ\varphi=g\circ\psi$ on a Zariski dense open subset of $Z$.
	
	By \cref{prop:crepant-blow-up}, there exist crepant morphisms
	$$
		\tilde{\varphi}\colon (\tilde{X},D_{\tilde{X}},\tilde{f})\to (X,D_X,f)
		\quad\text{and}\quad
		\tilde{\psi}\colon (\tilde{Y},D_{\tilde{Y}},\tilde{g})\to (Y,D_Y,g)
	$$
	such that $\tilde{f}$ and $\tilde{g}$ extend to proper morphisms $\tilde{X}\to\bb{P}^1_{\bb{C}}$ and $\tilde{Y}\to\bb{P}^1_{\bb{C}}$, respectively. 

	Consider the diagram
	$$
	\begin{tikzcd}
		& & \tilde{Z} \ar[ld]\ar[rd] &&\\
		& X' \ar[ld]\ar[rd]& & Y' \ar[ld]\ar[rd]& \\
		\tilde{X}\ar[rd,"\tilde{\varphi}"] & & Z\ar[ld,"\varphi"']\ar[rd,"\psi"] &&\tilde{Y}\ar[ld,"\tilde{\psi}"']\\
		& X &&Y&
	\end{tikzcd}
	$$
	where $X'$ maps birationally onto $\tilde{X}$ and $Z$, $Y'$ maps birationally onto $\tilde{Y}$ and $Z$, and $\tilde{Z}$ maps birationally onto $X'$ and $Y'$. Pulling back the log canonical divisors to $\tilde{Z}$ and using the crepancy of each map, one deduces that $(\tilde{X},D_{\tilde{X}},\tilde{f})$ and $(\tilde{Y},D_{\tilde{Y}},\tilde{g})$ are crepant to each other. 

	Therefore, by \cref{prop:crepant-blow-up} and \cref{prop:crepant-mor}, we have
	$$
	[X\setminus|D_X|,f]_{\bb{C}}=[\tilde{X}\setminus|D_{\tilde{X}}|,\tilde{f}]_{\bb{C}}=[\tilde{Y}\setminus|D_{\tilde{Y}}|,\tilde{g}]_{\bb{C}}=[Y\setminus|D_Y|,g]_{\bb{C}}
	$$
	in $\mathcal{E}xp\mathcal{M}_{\bb{C}}$.
\end{proof}

\begin{proof}[Proof of \upshape{\cref{intro::main}}]
	By \cref{thm::main}, the classes $[X\setminus|D_X|,f]_{\mathbb{C}}$ and $[Y\setminus|D_Y|,g]_{\mathbb{C}}$ are equal in $\mathcal{E}xp\mathcal{M}_{\mathbb{C}}$. Their associated exponential mixed Hodge structures are pure by \cref{prop:purity-EMHS}. The result follows from \cref{cor::same-Hodge-numbers}.
\end{proof}

\subsection{Applications}
\label{sec:Laurent-mutations-and-toric-Landau-Ginzburg-mirrors}
One direct application of \cref{intro::main} is the invariance under flops.
\begin{cor}\label{flops}
	For $i=1,2$, let
	$$
		\bar f_i\colon X_i\longrightarrow\bb{P}^1
	$$
	be a non-constant morphism from a smooth projective variety. Set
	$$
		D_i:=\bar f_i^*(\infty),\qquad
		U_i:=X_i\setminus|D_i|,
		\qquad f_i:=\bar f_i|_{U_i},
	$$
	and assume that $D_i$ has simple normal crossing support. If $X_1$ and $X_2$ are related by a sequence of flops over $\bb{P}^1$, then
	$$
		h^\alpha(U_1,f_1,k)=h^\alpha(U_2,f_2,k)
	$$
	for every $\alpha\in\bb{Q}$ and $k\in\bb{Z}$.
\end{cor}

\begin{proof}
	A sequence of flops gives a $K$-equivalence. Hence there exist a common smooth projective resolution and birational morphisms
	$$
		p\colon W\longrightarrow X_1,
		\qquad q\colon W\longrightarrow X_2
	$$
	such that $p^*K_{X_1}=q^*K_{X_2}$. Since the flops are over $\bb{P}^1$, the morphisms $\bar f_1\circ p$ and $\bar f_2\circ q$ agree. Therefore
	$$
		p^*D_1=(\bar f_1\circ p)^*(\infty)
		=(\bar f_2\circ q)^*(\infty)=q^*D_2,
	$$
	where the equality retains all multiplicities. It follows that
	$$
		p^*(K_{X_1}+D_1)=q^*(K_{X_2}+D_2).
	$$

	Choose a common finite regular value $c$ of $\bar f_1$ and $\bar f_2$. Then $\bar f_i^{-1}(c)$ is smooth and disjoint from $D_i$, so $(X_i,D_i,f_i-c)$ is a pair with a non-degenerate function. The two translated functions agree on a Zariski dense open subset of $W$. Thus the two pairs are crepant in the sense of \cref{defn:crepant}, and \cref{intro::main} gives the equality of their irregular Hodge numbers. Finally, replacing $f_i$ by $f_i-c$ changes neither $\rr{d}f_i$ nor its pole divisor. Hence the corresponding filtered twisted de Rham complexes, and therefore their irregular Hodge numbers, are unchanged.
\end{proof}
Another application comes from log Calabi--Yau varieties.
\begin{cor}\label{cor:birational-log-CY}
	Let $(X,D_X,f)$ and $(Y,D_Y,g)$ be two pairs with non-degenerate functions such that
	$$
		K_X+D_X=0
		\qquad\text{and}\qquad
		K_Y+D_Y=0.
	$$
	If $\varphi\colon X\dashrightarrow Y$ is a birational map such that $f=g\circ\varphi$ as rational functions, then
	$$
		h^\alpha(X\setminus|D_X|,f,i)
		=
		h^\alpha(Y\setminus|D_Y|,g,i)
	$$
	for every $\alpha\in\bb Q$ and $i\in\bb Z$.
\end{cor}

\begin{proof}
	Let $p\colon Z\to X$ and $q\colon Z\to Y$ be a common smooth projective resolution of $\varphi$. The log Calabi--Yau assumptions give
	$$
		p^*(K_X+D_X)=0=q^*(K_Y+D_Y).
	$$
	Moreover, $f\circ p=g\circ q$ as rational functions on $Z$. Hence $(X,D_X,f)$ and $(Y,D_Y,g)$ are crepant in the sense of \cref{defn:crepant}, and the conclusion follows from \cref{intro::main}.
\end{proof}

\begin{rmk} 
	\label{rmk:tame-versus-proper-compactifications}
	The log Calabi--Yau condition in \cref{cor:birational-log-CY} is a natural mirror symmetry assumption. It is precisely the log Calabi--Yau condition used in the toric Landau--Ginzburg literature \cite[Def.~2.4 and Rem.~2.5]{KatzarkovEtAlBogomolovTian17} \cite[Def.~3.6]{PrzyjalkowskiToricLG18}.
\end{rmk}

Mutations provide a natural source of birational log Calabi--Yau models. A \textit{mutation} of a Laurent polynomial $f$ on $T=\mathbb{G}_{m,\mathbb{C}}^n$ is a birational automorphism $\mu$ of $T$ which preserves the logarithmic volume form
$$
	\omega_T=\frac{\rr{d}x_1}{x_1}\wedge\cdots\wedge
	\frac{\rr{d}x_n}{x_n},
$$
such that $g=\mu^*f$ is again a Laurent polynomial; see \cite[Definition~1.1]{IltenMutations12}.

\begin{exe}\label{exe:ghv-mutations}
	Let $V\subset Y$ be a complete intersection in a toric Fano variety, cut out by sections of nef line bundles $L_1,\ldots,L_c$, and assume that $-K_Y-L_1-\cdots-L_c$ is ample. A nef partition for $(Y;L_1,\ldots,L_c)$, together with an amenable collection subordinate to it, gives a birational torus chart on the Givental--Hori--Vafa mirror and hence expresses its superpotential as a Laurent polynomial. The Newton polytope of this Laurent polynomial is reflexive \cite[Thms.~2.15 and~2.19]{DoranHarderToricDegenerations16}. Doran--Harder \cite[Thm.~2.24]{DoranHarderToricDegenerations16} prove that the Laurent polynomials obtained from two amenable collections subordinate to the same nef partition are related by a mutation. Prince \cite[Cor.~1.2]{PrinceBirationality21} removes the restriction that the nef partition be fixed: Laurent polynomial mirrors obtained from amenable collections subordinate to any nef partitions representing the same complete intersection are mutation equivalent. Thus all these Laurent polynomial presentations of the Givental--Hori--Vafa mirror lie in a single mutation class.
\end{exe}

\begin{cor}\label{cor:mutation-log-CY}
	Let $f$ and $g$ be mutation-equivalent Laurent polynomials on $T$. Suppose that they admit log Calabi--Yau compactifications $(X,D_X,F)$ and $(Y,D_Y,G)$ in the sense of \cref{defn:associated-compactification}. Then
	$$
		h^\alpha(X\backslash |D_X|,F,i)
		=
		h^\alpha(Y\backslash |D_Y|,G,i)
	$$
	for every $\alpha\in\bb Q$ and $i\in\bb Z$.
\end{cor}

\begin{proof}
	The mutations induce a birational map between $X$ and $Y$ which identifies $F$ and $G$ as rational functions. They satisfy the hypotheses of \cref{cor:birational-log-CY}, and hence have the same irregular Hodge numbers. 
\end{proof}

\begin{rmk}\label{rmk:ghv-log-CY-compactifications}
	The existence of log Calabi--Yau compactifications are known in many cases. For example, the smooth projective toric compactification associated to the refinement of a non-degenerate reflexive Laurent polynomial $f$ gives rise to such a compactification of $(T,f)$. Moreover, Przyjalkowski constructs them for Givental--Hori--Vafa toric Landau--Ginzburg models of Fano complete intersections in projective space \cite[Thm.~1]{PrzyjalkowskiCYCompactifications18}. Standard toric Landau--Ginzburg models of smooth Fano threefolds form another large class; see \cite[\S5.2]{PrzyjalkowskiToricLG18} and \cite[\S6]{DoranEtAlModularityLG23}.
\end{rmk}

\section*{Acknowledgement}
The author would like to thank Javier Fres\'an, Sukjoo Lee, Claude Sabbah, Christian Sevenheck, and Dingxin Zhang for valuable discussions.  

\bibliographystyle{alpha}
\bibliography{bibtex}

\end{document}